\documentclass[12pt]{amsart}
\usepackage{amssymb,amsmath,amsfonts,latexsym}
\usepackage{bm}
\usepackage[all,cmtip]{xy}
\usepackage{amscd}

\newcommand{\newsection}[1]{\setcounter{equation}{0} \section{#1}}
\newcommand{\bea}{\begin{eqnarray}}
\newcommand{\eea}{\end{eqnarray}}

\newcommand{\clb}{\mathcal{B}}

\newcommand{\cle}{\mathcal{E}}

\newcommand{\clh}{\mathcal{H}}

\newcommand{\clw}{\mathcal{W}}

\newcommand{\D}{\mathbb{D}}

\newcommand{\raro}{\rightarrow}

\newcommand{\mo}{\mathop{\oplus}}

\def\textmatrix#1&#2\\#3&#4\\{\bigl({#1 \atop #3}\ {#2 \atop #4}\bigr)}
\def\dispmatrix#1&#2\\#3&#4\\{\left({#1 \atop #3}\ {#2 \atop #4}\right)}
\newcommand{\be}{\begin{equation}}
\newcommand{\ee}{\end{equation}}
\newcommand{\ben}{\begin{eqnarray*}}
\newcommand{\een}{\end{eqnarray*}}

\newcommand{\NI}{\noindent}

\newcommand{\bi}{\begin{itemize}}
\newcommand{\ei}{\end{itemize}}

\newtheorem{Theorem}{\sc Theorem}[section]
\newtheorem{Lemma}[Theorem]{\sc Lemma}

\theoremstyle{definition}

\theoremstyle{plain}

\newtheorem{thm}{Theorem}[section]
\newtheorem{cor}[thm]{Corollary}
\newtheorem{lem}[thm]{Lemma}
\newtheorem{prop}[thm]{Proposition}
\theoremstyle{definition}

\newtheorem{rem}[thm]{Remark}
\newtheorem{ex}[thm]{Example}

\numberwithin{equation}{section}

\let\phi=\varphi

\begin{document}

\title[Numerical radius and Berezin number inequality]
{Numerical radius and Berezin number inequality}

\author[Majee]{Satyabrata Majee}
\address{Indian Institute of Technology Roorkee, Department of Mathematics,
		Roorkee-247 667, Uttarakhand,  India}
\email{smajee@ma.iitr.ac.in}

\author[Maji]{Amit Maji}
\address{Indian Institute of Technology Roorkee, Department of Mathematics,
		Roorkee-247 667, Uttarakhand,  India}
\email{amit.maji@ma.iitr.ac.in, amit.iitm07@gmail.com}

\author[Manna]{Atanu Manna}
\address{Indian Institute of Carpet Technology, Bhadohi-221401, Uttar Pradesh, India}
\email{ atanu.manna@iict.ac.in, atanuiitkgp86@gmail.com}
	




\begin{abstract}
We study various inequalities for numerical radius and Berezin number 
of a bounded linear operator on a Hilbert space. It is proved that the numerical 
radius of a pure two-isometry is $1$ and the Crawford number of a pure 
two-isometry is $0$. In particular, we show that for any scalar-valued
non-constant inner function $\theta$, the numerical radius and the 
Crawford number of a Toeplitz operator $T_{\theta}$ on a Hardy space 
is $1$ and $0$, respectively. It is also shown that numerical radius is 
multiplicative for a class of isometries and sub-multiplicative for a 
class of commutants of a shift. We have illustrated these results 
with some concrete examples. 
Finally, some Hardy-type inequalities for Berezin number of certain 
class of operators are established with the help of the classical 
\textit{Hardy's inequality}.
\end{abstract}
\maketitle

\newsection{Introduction}

The concepts of numerical radius and Berezin number of an operator 
have been studied extensively due to their enormous applications 
in engineering, quantum computing, quantum mechanics,
numerical analysis, differential equations, etc.

In the early days, one of the main goals for studying Hilbert spaces 
was quadratic forms. The fundamental question about quadratic forms 
associated with an operator is its numerical range. Firstly, Toeplitz 
\cite{Toeplitz} defined the numerical range for matrices in 1918. Later, Wintner 
\cite{Wintner} studied the relationship between the numerical range and the 
convex hull of the spectrum of a bounded linear operator on a Hilbert space.

We denote $\clb(\clh)$ as the $C^{*}$-algebra of all bounded linear 
operators on a Hilbert space $\clh$. For any bounded linear operator $A$ on 
$\clh$ the numerical radius, denoted by $\omega(A)$, yields a norm on 
$\clb(\clh)$. Indeed, for $A \in \clb(\clh)$
\[
{\|A\| \over 2} \leq \omega(A) \leq \|A\|.
\]
Therefore, the operator norm and the numerical radius norm are equivalent.
In particular, if $A$ is normal (that is, $A^*A=AA^*$), then
$\omega(A)=\|A\|$. One can show that the numerical radius $\omega(\cdot)$ is  
neither multiplicative nor sub-multiplicative on $\clb(\clh)$(see, \cite{Halmos}). 
Thus a natural question arises:
\begin{center}
$(Q)$ \textsf{Which class of operators on a Hilbert space $\omega(\cdot)$ is 
multiplicative or sub-multiplicative?}
\end{center}

Though numerical radius is not multiplicative or sub-multiplicative 
in general, but it satisfies the power inequality. More precisely, for any 
positive integer $n$ and $A \in \clb(\clh)$ the following inequality holds:
\begin{center}
$\omega(A^n) \leq \omega(A)^n$.
\end{center}
The above power inequality was first conjectured by Halmos and a delicate 
proof was given by Berger (see \cite{BS-Mapping}, \cite{Halmos}). After that, 
generalizations for polynomial or analytic functions on the unit disc have been done.
In 1966, Pearcy \cite{PEARCY} gave elementary proof of this. However, the 
reverse power inequality does not hold in general. Indeed, for a nilpotent matrix
$
A= \begin{pmatrix}
0 & 0 \\
1 & 0
\end{pmatrix},
$
we have $\omega(A)=\frac{1}{2}$ but $\omega(A^2)=0$.
The reader is referred to \cite{BD-Numerical-I}, \cite{GR-Book} and 
references therein for various applications of numerical radius inequalities.

On the other hand, Berezin transform which connects 
operators with functions plays an important role in operator theory, 
more specifically to study Toeplitz operators, Hankel operators
and composition operators. 
Another important notion in operator theory is Berezin number.
The Berezin number of an operator $A$ is denoted by $ber(A)$ and
defined on a reproducing kernel Hilbert space  
$H(\Omega)$. Indeed, for any $A \in \clb( H(\Omega))$,
$ 0 \leq ber(A) \leq \omega(A) \leq \|A \| $  
and if $ber(A)\neq 0$, then the following inequality for $n \in \mathbb{N}$
(see Garayev et al. \cite{GARAGUR}) 
\[
ber(A^n) \leq (\frac{w(A)}{ber(A)})^n ber(A)^n \quad \mbox{holds}.
\]
A lot of research work has been carried out to find the reverse power inequality
of the Berezin number of operators over the last few years and many researchers 
have studied the reverse inequality by using Hardy-Hilbert inequality. Apart from 
these studies, Yamanc{\i} et al.\cite{YAMAGUR} considered Berezin number operator 
inequalities for continuous convex functions and some certain class of operators.   
For more details one can see Karaev(\cite{KARA}, \cite{KARA2}),
Garayev et al. (\cite{GARAGUR}, \cite{GARASAL}), 
Yamanc{\i} et al.\cite{YAMA} and references therein. 
However, Coburn \cite{COBURN} gave an example of concrete operator 
$A$ on the Bergman space such that $ber(A^2)>ber(A)^2$ holds. 
Motivated by the earlier studies, we here attempt to explore
the following inequality:
\begin{center}
\textsf{$ ber(A)^n\leq \alpha ber(A^n)$ for $A\in \clb(H(\Omega)$ 
all integer $n>1$ and some constant $\alpha>0$}.
\end{center}

The main guiding tools used in this article are the geometry of 
Hilbert spaces, Wold-von Neumann decomposition for pure two-isometry and 
isometry. Using the notion of functional calculus, some reverse 
power inequalities for Berezin number are also established by using the classical 
\emph{discrete Hardy's inequality} (See Section 2, (\ref{Hardy-Inequality})).
\vspace{.1cm}

The paper is organized as follows: in Section 2 we discuss 
preliminaries and some basic results. Section 3 deals with 
the numerical radius and Crawford number of certain class 
of operators and various results on the Berezin number. 
In section 4 with the help of \emph{discrete Hardy's inequality} 
we develop the Berezin number inequality which improves many earlier results.

\section{Preliminaries and Basic Results}

Let $p>1$ be any real number and $\{a_n\}$ be a sequence of non-negative 
real numbers. Then the discrete version of \emph{Hardy's inequality} 
(named after G. H. Hardy) is
\begin{align}{\label{Hardy-Inequality}}
\displaystyle\sum_{n=1}^{\infty}\Big(\frac{1}{n}\sum_{k=1}^{n}a_k\Big)^p
\leq \Big(\frac{p}{p-1}\Big)^p\sum_{n=1}^{\infty}{a_n}^p,
\end{align}
and equality holds if $a_n=0$ for all $n$. The constant term 
$\Big(\frac{p}{p-1}\Big)^{p}$ in (\ref{Hardy-Inequality}) is the best 
possible constant. There are several generalizations and extensions
which have been studied by many researchers in different scientific works
(for more details, see \cite{HARDLITTPOLY}).

For $p=2$ the above inequality (\ref{Hardy-Inequality}) becomes
\begin{align}{\label{Hardy-Inequality-Modified}}
\displaystyle\sum_{n=1}^{\infty} w_n^{H} \left(\sum_{k=1}^{n}a_k \right)^2
\leq \sum_{n=1}^{\infty}{a_n}^2,
\end{align}
where $w_n^{H} = \frac{1}{4n^2}$. We say $w_n^{H}$ is the classical Hardy weight.
Recently, Keller et al. \cite{KELLER} obtained an improved version of discrete
Hardy inequality for $p=2$ as follows
\begin{align}{\label{imph}}
\displaystyle\sum_{n=1}^{\infty}w_n\Big(\sum_{k=1}^{n}a_k\Big)^2\leq\sum_{n=1}
^{\infty}{a_n}^2,
\end{align}
where $w_n = 2-\sqrt{1+\frac{1}{n}}-\sqrt{1-\frac{1}{n}}$ for $n\in \mathbb{N}$.
Clearly, $w_n>\frac{1}{4n^2}=w_n^{H}$ for each $n\in \mathbb{N}$. Also note that 
$\displaystyle \sum_{n=1}^{\infty}w_n^{H}= \frac{\pi^2}{24}$ and the series 
$\displaystyle \sum_{n=1}^{\infty}w_n$ is convergent and converges to $0.753045$ 
(approximately).

In what follows, $\clh$ stands for separable complex Hilbert space.
The set of all bounded linear operators from $\clh$ to itself is denoted 
by $\clb(\clh )$. An operator $A\in \mathcal{B}(\mathcal{H})$ is said 
to be an isometry if $A^*A=I_{\clh}$ and $A$ is said to be a pure 
isometry or a shift if $A^{*n} \rightarrow 0$ as $n \rightarrow \infty$ 
in strong operator topology (cf. \cite{Halmos}). We say an operator $A$ 
is a co-isometry if $A^*$,
adjoint of $A$ is an isometry and $A$ is a co-shift if $A^*$ is a shift. 
An operator $A \in \clb(\clh)$ is said to be positive if $A$ is self-adjoint, 
that is, $A^* = A$ and
\[
\langle Ax, x \rangle \geq 0 \quad (x \in \clh).
\]
The numerical range of any $A\in \mathcal{B}(\mathcal{H})$, denoted by 
$W(A)$, is defined as
\[
W(A) :=\{\langle Ax,x \rangle : x\in \clh ,\|x\|=1 \},
\]
the numerical radius of $A$ and the Crawford number of $A$, denoted by 
$\omega(A)$ and $c(A)$, respectively as
\[
\omega(A)=\sup\{ |\langle Ax,x \rangle| : x\in \clh ,\|x\|=1  \},
\]
and
\[
c(A)=\inf\{ |\langle Ax,x \rangle| : x\in \clh ,\|x\|=1  \}.
\]
We say two operators $A_1$ on $\clh_1$ and $A_2$ on $\clh_2$ are said 
to be unitarily equivalent if there exists a unitary $U: \clh_1 \rightarrow 
\clh_2$ such that $UA_1 = A_2U$. From the definition it readily follows 
that the numerical radius and the Crawford number are unitarily invariant. 
Indeed, if $A_1$ on $\clh_1$ and $A_2$ on $\clh_2$ are unitarily equivalent,
then $\omega(A_2)=\omega(UA_1U^{*}) = \omega(A_1)$ and $c(A_2)=c(UA_1U^{*}) 
= c(A_1)$. Also $\omega(A)= \omega(A^*)$ and $c(A)=c(A^{*})$ for any $A \in \clb(\clh)$.

Let $\Omega$ be a non-empty set. A reproducing kernel Hilbert space (in short rkHs)
is a Hilbert space $H(\Omega)$ of complex-valued functions on $\Omega$
such that for every point $w \in \Omega$ the point evaluation $f \mapsto f(w)$ 
is a bounded linear functional on $H(\Omega)$. Since for each $w \in \Omega$ 
the map $f \mapsto f(w)$ is a continuous linear functional on $H(\Omega)$, 
there is a unique element $k_w$ of $H(\Omega)$ because of the classical Riesz's
representation theorem such that $f(w)=\langle f, k_w \rangle$ for all 
$f\in H(\Omega)$. The map $k: \Omega \times \Omega \rightarrow \mathbb{C} $ 
defined by
\[
k(z,w) = k_w(z)= \langle k_w, k_z \rangle
\]
is called the reproducing kernel function of $H(\Omega)$. We denote $\hat{k}_w = 
\frac{{k}_w}{\|{k}_w\|}$ for $w \in \Omega$ as the normalized reproducing 
kernel of $H(\Omega)$ and the set $\{\hat{k}_w: w \in \Omega \}$ is a total 
set in $H(\Omega)$. If a sequence $\{ e_n \}$ is an orthonormal basis for 
$H(\Omega)$, then the kernel function can be written as 
\begin{center}
$k(z, w) =\displaystyle\sum_{n=0}^{\infty}e_n(z)\overline{e_n(w)}$.
\end{center}
For more details and references on rkHs, see \cite{Aronzajn-rkHs},
\cite{PR-rkHs-Book}. Let $H(\Omega)$ be a reproducing kernel 
Hilbert space on $\Omega$ and let $A \in \clb( H(\Omega))$.
Define a function $\widetilde{A}$ on $\Omega$ as
\[
\widetilde{A}(z)=\langle A\hat{k}_z, \hat{k}_z\rangle, \quad (z \in \Omega).
\]
The function $\widetilde{A}(z)$ is called Berezin transform or Berezin symbol
of $A$ which was firstly introduced and studied by Berezin (\cite{BERE}, 
\cite{BEREQUAN}). Using Cauchy-Schwarz inequality, one can easily say that 
the Berezin transform $\widetilde{A}$ is a bounded function on $\Omega$. Indeed,
\[
|\widetilde{A}(z)|\leq \|A \hat{k}_z\| \|\hat{k}_z\| \leq \|A \| \quad (z \in \Omega).
\]
The Berezin set and Berezin number of an operator $A \in \clb( H(\Omega))$,
denoted by $Ber(A)$ and $ber(A)$, respectively is defined as
\begin{align*}
Ber(A)& = \{\widetilde{A}(z): z \in \Omega\},\\
ber(A)& =\displaystyle\sup_{z\in \Omega}|\widetilde{A}(z)|.
\end{align*}
Therefore, one can easily observe that for $A \in \clb( H(\Omega))$
\begin{center}
$0 \leq ber(A) \leq w(A)\leq \|A\|$.
\end{center}
From the definition of Berezin number, it readily follows that
$ber(\cdot)$ is a semi-norm on $\clb( H(\Omega))$.

Before proceeding further, let us recall the notion of functional 
calculus \cite{Conway-Book}. Let $A$ be a normal operator on a Hilbert space $\clh$ 
and $\sigma(A)$ denote the spectrum of $A$. Let $\mathcal{C}(\sigma(A))$ 
and $B_{\infty}(\sigma(A))$ be the continuous complex-valued functions 
and the bounded measurable complex-valued functions on $\sigma(A) 
\subseteq \mathbb{C}$, respectively. Then $B_{\infty}(\sigma(A))$ 
is a $C^*$-algebra with the involution map defined by $f \mapsto \bar{f}$. 
Let $\pi: \mathcal{C}(\sigma(A)) \rightarrow \clb(\clh)$ be a $*$-homomorphism 
such that $\pi(1)= I_{\clh}$, where $1$ is a constant function with value one 
and $I_{\clh}$ is an identity operator. Then there exists a unique spectral measure 
$\mathcal{P}$ in $(\sigma(A), \clh)$ such that
\[
\pi(f) = \int_{\sigma(A)}fd\mathcal{P}.
\]
Now the continuous functional can be extended by Borel functional calculus
for any $f \in B_{\infty}(\sigma(A))$, where
\[
f(A) = \displaystyle\int_{\sigma(A)}fd\mathcal{P}.
\]
In particular, if $A \in \clb(\clh)$ is a positive operator, then
$\sigma(A)$ is a subset of $[0, \infty)$ and $f(A)$ is a positive
operator for any positive Borel function whose domain contains the
spectrum of $A$.

The following inequalities are useful to prove our results.

\begin{Lemma}(\cite{SIMON-Trace})\label{Macarthy-lemma}
Let $A$ be any positive operator on a Hilbert space $\clh$ and $p > 0$
be any real number. Then for $x \in \clh$,
\begin{center}
$\langle Ax, x \rangle^p \leq \langle A^{p}x, x \rangle$ whenever $1 \leq p < \infty$
\end{center}
and
\begin{center}
$\langle Ax, x \rangle^p > \langle A^{p}x, x \rangle$ whenever $0 < p < 1$.
\end{center}
\end{Lemma}

\begin{Lemma}(\cite{KATO-Notes})\label{Convex-lemma}
Let $A$ be any bounded linear operator on a Hilbert space $\clh$. 
Then for any $x, y\in \clh$
\begin{center}
$|\langle Ax, y \rangle|^2 \leq \langle |A|^{2\alpha}x, x \rangle
\langle |{A^*}|^{2(1-\alpha)}y, y \rangle$,
\end{center}
where $0\leq \alpha \leq 1$.
\end{Lemma}

\section{Numerical Radius and Berezin number}

In general, it is very difficult to compute Berezin number and numerical 
radius of a bounded operator (even for particular classes of operators).
But if we concentrate on the model space, then these values can be 
calculated for certain classes of operators. In this section, our treatment 
is analytic in nature, and we also discuss some inequalities and concrete 
examples to find Berezin number, Crawford number, and  numerical radius.
\vspace{.1cm}

Let $A \in \clb(\clh)$. The positive square root of $A$, denoted by $|A|$,
is defined as $|A|=\sqrt{A^*A}$ is a self-adjoint operator on $\clh$.
Now for any $x \in \clh$,
\[
\| |A| \|^2 = \sup_{\|x \| =1}\||A|x \|^2 = \sup_{\|x \| =1} 
|\langle A^*Ax, x \rangle| = \|A^*A \|=\|A\|^2.
\]
Therefore
\[
\|| A | \| = \|A \|= \| A^{*} \| = \| | A^{*} | \|.
\]
Since $|A|$ and $| A^{*} |$ are self-adjoint operators,
\[
\omega(|A|) = \||A|\| =  \| A \| = \| | A^{*} | \| =\omega(|A^{*}|).
\]
Hence for any operator $A\in \clb(\clh)$,
\[
\omega(A) \leq \omega(|A|).
\]
Now for any real number $p>0$ and $x \in \clh$
with $\|x \| =1$,
\begin{align*}
|\langle Ax, x \rangle|^{p}  \leq \|Ax \|^{p} = \langle Ax, Ax \rangle^{p \over 2} 
= \langle A^*Ax, x \rangle^{p \over 2}. 
\end{align*}
Therefore
\[
 \sup_{\|x \|=1} |\langle Ax, x \rangle|^{p} \leq 
\sup_{\|x \|=1} \langle A^*Ax, x \rangle^{p \over 2}= \|A^*A \|^{p \over 2}= \|A \|^p 
= \||A|\|^p= \||A|^p\|,
\]
where the last equality follows from functional calculus for the positive operator $|A|$.
Hence 
\[
\omega(A)^{p} \leq \omega(|A|^{p}).
\]
From the above deliberation we have the following:

\begin{lem}
Let $A \in \clb(\clh)$ and $p>0$ be any real number. Then
\[
\omega(A)^{p} \leq \omega(|A|^{p}).
\]
\end{lem}

\begin{thm}
Let $A_i, X \in \clb(\clh)$ for $i=1, \ldots, n$ and $X$ 
be any positive operator. Then 
\[
\omega({\sqrt{\sum_{i=1}^n A_i^*XA_i}}) = {\sqrt{\omega(\sum_{i=1}^n A_i^*XA_i)}}.
\]
\end{thm}

\begin{proof}
Suppose that $A_i, X \in \clb(\clh)$ for $i=1, \ldots, n$ and $X$ 
is any positive operator. Then $\sum_{i=1}^n A_i^*XA_i$ is a positive operator
as for any $h \in \clh$
\begin{align*}
\langle (\sum_{i=1}^n A_i^*XA_i) h, h \rangle  =  \sum_{i=1}^n \langle A_i^*XA_i h, h \rangle 
=  \sum_{i=1}^n \langle XA_i h, A_i h \rangle \geq 0.
\end{align*}
Therefore
\begin{align*}
\omega({\sqrt{\sum_{i=1}^n A_i^*XA_i}})  = \sup_{\|h \|=1} 
|\langle {\sqrt{\sum_{i=1}^n A_i^*XA_i}}h,  h \rangle| 
& = \| {\sqrt{\sum_{i=1}^n A_i^*XA_i}} \| \\
& = \| {\sum_{i=1}^n A_i^*XA_i} \|^{1 \over 2} \\
& = (\sup_{\|h \|=1} |\langle {\sum_{i=1}^n A_i^*XA_i}h,  h \rangle|)^{1 \over 2} \\
& = \sqrt{\omega(\sum_{i=1}^n A_i^*XA_i)}.
\end{align*}
Here the above third equality follows from functional calculus 
for a positive operator.
This proves the desire inequality.
\end{proof}

The numerical radius $\omega(\cdot)$ is a norm on $\clb(\clh)$ but $\omega(\cdot)$ is 
not multiplicative, that is, $\omega(RT) \neq \omega(R)\omega(T)$  
(even for normal operators) $R, T \in \clb(\clh)$. For example 
we consider $T, R \in \clb(\mathbb{C}^2)$, defined as $T(z, w)= (z, 0)$ 
and $R(z, w)= (0, w)$. Then $\omega(T)=\|T \|=1 =\|R \|=\omega(R)$ but $\omega(RT)=0$.
Since ${\|T \| \over 2} \leq \omega(T) \leq \|T \|$ for $T\in \clb(\clh)$ 
(see \cite{Halmos}),
\[
\omega(RT) \leq 4 \omega(R) \omega(T)
\] 
for any $R, T \in \clb(\clh)$. For a normal operator $T$, $\omega(T) = \|T \|$ 
and hence if $R, T$ are normal, then 
\[
\omega(RT) \leq \omega(R)\omega(T).
\]
Therefore $\omega(\cdot)$ is sub-multiplicative in case of normal operators,
but not in general. Here we attempt to give an answer (partially) to 
the natural question raised in Section 1.

We will firstly study model spaces and discuss the classical Wold-von 
Neumann decomposition for an isometry (see \cite{NF-Book}, \cite{MSS-PAIRS}).

\begin{thm}\label{thm-Wold} Let $V$ be an isometry on a Hilbert space $\clh$.
Then $\clh$ decomposes as a direct sum of $V$-reducing subspaces
$\clh_s = \displaystyle{\mo_{m=0}^\infty} V^m \clw$ and
$\clh_u = \clh \ominus \clh_s$ and
\[
V = \begin{bmatrix} V_s & 0\\ 0 & V_u
\end{bmatrix} \in \clb(\clh_s \oplus \clh_u),
\]
where $\clw = \clh \ominus V\clh$, $V_s = V|_{\clh_s}$ is a shift operator 
and $V_u = V|_{\clh_u}$ is a unitary operator.
\end{thm}

The \textit{Hardy space}, denoted by $H^2(\D)$, is the Hilbert space
of all square summable holomorphic functions on the unit disc $\D$
(cf. \cite{NF-Book}). The Hardy space is also a reproducing
kernel Hilbert space corresponding to the Szeg\"{o} kernel
\[
k(z, w) = \frac{1}{1 - z \bar{w}} \quad \quad (z, w \in \D).
\]
For any Hilbert space $\cle$, $H^2_{\cle}(\D)$ denotes the $\cle$-valued 
Hardy space with reproducing kernel $k(z, w)I_{\cle}$, and
$H^\infty_{\clb(\cle)}(\D)$ denotes as the space of $\clb(\cle)$-valued 
bounded holomorphic functions on $\D$. The multiplication operator $M_z$ 
by the coordinate function $z$ on $H^2_{\cle}(\D)$ is defined by
\[
(M_z f)(w) = w f(w) \quad \quad \quad (f \in H^2_{\cle}(\D),
w \in \D).
\]

The Wold-von Neumann construction yields an analytic description 
of an isometry as follows: Let $V$ be an isometry on $\clh$, and let 
$\clh = \clh_s \oplus \clh_u$ be the Wold-von Neumann orthogonal 
decomposition of $V$. Define
\[
\Pi : \clh_s \oplus \clh_u \raro H^2_{\clw}(\D) \oplus \clh_u
\]
by
\[
\Pi (V^m \eta \oplus f) = z^m \eta \oplus f \quad \quad (m \geq 0,
f \in \clh_u, \eta \in \clw).
\]
Then $\Pi$ is a unitary and
\[
\Pi (V_s \oplus V_u) = (M_z \oplus V_u)\Pi,
\]
that is, $V$ on $\clh$ and $ M_z \oplus V_u $ on $H^2_{\clw}(\D) 
\oplus \clh_u$ are unitarily equivalent. In particular, if $V$ is a shift,
then $\clh_u =\{0\}$ and hence
\[
\Pi V = M_z\Pi.
\]
Therefore, an isometry $V$ on $\clh$ is a shift operator if and only
if $V$ is unitarily equivalent to $M_z$ on $H^2_{\cle}(\D)$, where 
$\dim \cle = \dim (\clh \ominus V\clh)$. We call $(\Pi, M_z)$ 
as the Wold-von Neumann decomposition of the isometry $V$.
Let $C$ be a bounded linear operator on $H^2_{\cle}(\D)$. Then 
$C \in \{M_z\}^{'}$, commutant of $M_z$ (that is, $CM_z = M_z C$), 
if and only if (cf. \cite{NF-Book})
\[
C = M_{\Theta},
\]
for some $\Theta \in H^{\infty}_{\clb(\cle)}(\D)$, where
$(M_{\Theta}f)(w) = \Theta(w) f(w)$ for all $f \in H^2_{\cle}(\D)$ 
and $w \in \D$. Let $V$ be a pure isometry on a Hilbert space $\clh$, 
and let $C \in \{V\}^{'}$. Suppose that $(\Pi, M_z)$ is the Wold-von 
Neumann decomposition of the pure isometry $V$ on $\clh$, and 
$\clw = \clh \ominus V\clh$. Then $\Pi C \Pi^* \in \clb(H^2_{\clw}(\D))$ 
and $(\Pi C \Pi^*) M_z = M_z(\Pi C \Pi^*)$ as $CV=VC$. Therefore
\[
\Pi C \Pi^* = M_{\Theta}
\]
for some  $\Theta \in H^{\infty}_{\clb(\clw)}(\D)$.

Another important vector-valued reproducing kernel Hilbert space is the
Dirichlet space $D_{\cle}(\mu)$ which is defined as the
space of all $\cle$-valued holomorphic functions $f$ in the
unit disc $\D$ with the finite norm
\[
\|f\|_{\mu}^2= \|f \|^2_{H^2(\D)} + \int_{\D}\langle P[\mu](z)f^{'}(z), f^{'}(z) \rangle dA(z), 
\] 
where $dA$ stands for the normalized Lebesgue area measure on $\D$
and $ P[\mu]$ is the Poisson integral of the positive $\clb(\cle)$-valued
operator measure on the unit circle. The multiplication operator
$M_z^{D_{\cle}(\mu)}$ stands for the Dirichlet shift on the space
$D_{\cle}(\mu)$ (see \cite{{Olofsson}} for more details).
An operator $V \in \clb(\clh)$ is said to be a two-isometry if
\[
\|V^2h\|^2 + \|h \|^2= 2\|Vh\|^2 \quad (\forall~ h \in \clh).
\]
Thus every isometry is a two-isometry. Let $V$ be a pure two-isometry
on $\clh$, that is $\clh_u= \cap_{n=0}^{\infty}V^n\clh = \{0\}$. Then 
the Wold-von Neumann decomposition for the two isometry gives an 
analytic representation. More precisely, the unitary 
$\tilde{\Pi}: \clh \rightarrow D_{\cle}(\mu)$ defined by 
\[
(\tilde{\Pi}h)(z) = \sum_{n=0}^{\infty}(P_{\cle}L^nh)z^n \quad (h \in \clh, z \in \D),
\]
where $P_{\cle}$ is the projection of $\clh$ onto $\cle = \clh \ominus V\clh$,
$L=(V^*V)^{-1}V^*$ is the left inverse of $V$, and $\tilde{\Pi}V=M_z^{D_{\cle}(\mu)}\tilde{\Pi}$.

We are now in a position to state our results:

\begin{thm}\label{Shift-Numerical}
Let $V$ be a shift or a co-shift on a Hilbert space $\clh$. Then
the Crawford number of $V$ is $0$.
\end{thm}

\begin{proof}
Suppose that $V$ is a shift on a Hilbert space $\clh$. We here 
prove only for shift case. If $V$ is a co-shift, then we take 
adjoint of $V$, $V^*$ which becomes a shift.

Let $(\Pi, M_z)$ be the Wold-von Neumann decomposition of the shift $V$.
That means the map $\Pi: \clh \rightarrow H^2_{\cle}(\D)$, where 
$\cle =\clh \ominus V\clh$ is a unitary such that
\[
\Pi V \Pi^* = M_z.
\]
Therefore
\begin{align*}
c(M_z) = c(\Pi V \Pi^*)= c(V) \leq \|V\| =1.
\end{align*}
Now consider the normalized kernel function $\hat{k}(\cdot, w) 
= \frac{{k}(\cdot, w)}{\|{k}(\cdot, w)\|}$ of $H^2(\D)$. Also
the set $\{\hat{k}(\cdot, w)\eta : w \in \D, \eta \in \cle \}$ 
is a total set in $H^2_{\cle}(\D)$. Therefore 
from the definition of Crawford number of an operator, 
we obtain
\begin{align*}
0 \leq c(M_z ) & \leq  \inf_{w \in \D, \|\eta\| =1} |\langle M_z\hat{k}
(\cdot, w) \eta, \hat{k}(\cdot, w)\eta \rangle |\\
& =  \inf_{w\in \D, \|\eta\| =1} |\langle \hat{k}(\cdot, w)\eta , 
M_z ^{*} \hat{k}(\cdot, w)\eta  \rangle |\\
& =  \inf_{w \in \D, \|\eta\| =1} |\bar{w} | |\langle \hat{k}(\cdot, w)
\eta,  \hat{k}(\cdot, w)\eta  \rangle |\\
& = \inf_{w \in \D} |{w} |\\
& =0.
\end{align*} 
Thus
\[
c(V^{*}) = c(V) = c(M_z ) = 0.
\]
This completes the proof.
\end{proof}

\begin{rem}
Let $\theta \in H^{\infty}(\D)$ be a non-constant inner function.
Then the analytic Toeplitz operator $T_{\theta} \in \clb(H^2(\D))$ is a shift 
(cf. \cite{KDPS-Partially}). Hence 
\[
\omega(T_{\theta}^{*n}) = \omega(T_{\theta}^n) = 1 
\]
and 
\[
c(T_{\theta}^{*n}) = c(T_{\theta}^n) = 0 \quad (n \in \mathbb{N}).
\]
\end{rem}

\begin{rem}
Suppose $\theta \in H^{\infty}(\D)$. If $\theta(z_0)=0$ for 
some point $z_0 \in \D$, then 
\[
c(T_{\theta}^n) = 0 \quad (n \in \mathbb{N}).
\]
For an example, consider $\theta(z)= 1$ for $z \in \D$. Then 
clearly $\theta \in H^{\infty}(\D)$ and $\theta(z) \neq 0$ for 
any $z$ in $\D$. Also $c(T_{\theta}^n) = 1 $ for any $n \in \mathbb{N}$.
\end{rem}

\begin{thm}\label{two-isometry-Numerical}
Let $V$ be a pure two-isometry on a Hilbert space $\clh$. Then
the numerical radius of $V$ is $1$ and the Crawford number of $V$
is $0$.
\end{thm}

\begin{proof}

Suppose that $V$ is a pure two-isometry on $\clh$. 
Let $(\tilde{\Pi}, M_z^{D_{\cle}(\mu)})$ be the Wold-von Neumann 
decomposition of $V$. Then 
$
\tilde{\Pi} V \tilde{\Pi}^{*} = M_z^{D_{\cle}(\mu)}.
$
Thus
\begin{align*}
\omega(M_z^{D_{\cle}(\mu)}) = \omega(\tilde{\Pi} V \tilde{\Pi}^{*}).
\end{align*}
Again consider the normalized kernel function $\hat{k}(\cdot, w) 
= \frac{{k}(\cdot, w)}{\|{k}(\cdot, w)\|}$ of $D(\mu)$. Also
the set $\{\hat{k}(\cdot, w)\eta : w \in \D, \eta \in \cle \}$ 
is a total set in $D_{\cle}(\mu)$. Therefore
\begin{align*}
\omega(M_z^{D_{\cle}(\mu)}) 
& =  \sup_{w\in \D, \|\eta\| =1} |\langle M_z^{D_{\cle}(\mu)}\hat{k}(\cdot, w)\eta, 
\hat{k}(\cdot, w)\eta \rangle |\\
& =  \sup_{w \in \D, \|\eta\| =1} |\bar{w} | |\langle \hat{k}(\cdot, w)
\eta,  \hat{k}(\cdot, w)\eta  \rangle |\\
& = \sup_{w \in \D} |{w} |\\
& =1.
\end{align*} 
Hence
\[
\omega(V) = \omega(\tilde{\Pi} V \tilde{\Pi}^{*}) =\omega(M_z^{D_{\cle}(\mu)})=1.
\]

For the case of Crawford number, the same lines of the above proof yields
\[
c(V) = c(\tilde{\Pi} V \tilde{\Pi}^{*}) = c(M_z^{D_{\cle}(\mu)})= \inf_{w \in \D}|w|=0. 
\]

This completes the proof.
\end{proof}

\begin{rem}\label{Isometry-Numerical}
Let $V$ be an isometry on a Hilbert space $\clh$. Then
it is a two-isometry with $\|V\|=1$. Now the spectral radius 
formula gives that the numerical radius of $V$ is $1$. However, here we
provide a new proof by the help of the analytic structure of an isometry. 
Clearly,
\begin{align*}
\omega(V) = \sup_{\|x \| =1}|\langle Vx, x \rangle| \leq \| V \|=1.
\end{align*}
Let $(\Pi, M_z)$ be the Wold-von Neumann decomposition of the isometry $V$.
Then 
\[
\Pi V \Pi^* = M_z \oplus U,
\]
where $U$ is the unitary part of $V$. Thus
\begin{align*}
\omega(M_z \oplus U) = \omega(\Pi V \Pi^*)= \omega(V) \leq 1.
\end{align*}
Again consider the normalized kernel function $\hat{k}(\cdot, w) 
= \frac{{k}(\cdot, w)}{\|{k}(\cdot, w)\|}$ of $H^2(\D)$. Also
the set $\{\hat{k}(\cdot, w)\eta : w \in \D, \eta \in \cle \}$ 
is a total set in $H^2_{\cle}(\D)$. Therefore
\begin{align*}
\omega(M_z \oplus U) 
& \geq   \sup_{w\in \D, \|\eta\| =1} |\langle(M_z \oplus U)(\hat{k}(\cdot, w)\eta 
 \oplus 0), \hat{k}(\cdot, w)\eta  \oplus 0 \rangle |\\
& =  \sup_{w\in \D, \|\eta\| =1} |\langle \hat{k}(\cdot, w)\eta , 
M_z ^{*} \hat{k}(\cdot, w)\eta  \rangle |\\
& =  \sup_{w \in \D, \|\eta\| =1} |\bar{w} | |\langle \hat{k}(\cdot, w)
\eta,  \hat{k}(\cdot, w)\eta  \rangle |\\
& = \sup_{w \in \D} |{w} |\\
& =1.
\end{align*} 
Hence
\[
 \omega(V) = \omega(\Pi V \Pi^*) = \omega(M_z \oplus U) = 1.
\]
If $V$ is a co-isometry, then $V^*$ is an isometry. Thus from the 
above we have
\[
\omega(V) = \omega(V^*)=  1.
\]

The above analytic proof yields that the numerical radius of a two-isometry
with unit norm is $1$.
\end{rem}

\begin{cor}
The numerical radius $\omega(\cdot)$ is multiplicative on the class
of isometries on a Hilbert space.
\end{cor}

\begin{proof}
Let $R, T$ be two isometries on a Hilbert space $\clh$. Then 
$RT$ is an isometry on $\clh$ as 
\[
(RT)^*(RT)= T^*(R^*R)T = (T^*I_{\clh})T = T^*T = I_{\clh}.
\]
Therefore from the above Remark \ref{Isometry-Numerical}, we have
\[
\omega(RT) = 1 ~~~ \mbox{and}~~~ \omega(R)= 1 = \omega(T).
\]
Hence
\[
\omega(RT) = \omega(R) \omega(T).
\]
This finishes the proof.
\end{proof}

\begin{rem}
Let $T \in \clb(\clh)$ be any isometry. Then the above Corollary
gives $\omega(T^n)= \omega(T)^n$ for any $n \in \mathbb{N}$. It 
is a noteworthy to mention that for a normal operator $T \in 
\clb(\clh)$, $\omega(T^n)= \omega(T)^n$ for any $n \in 
\mathbb{N}$ as
\[
\omega(T^n) = \|T^n\| = \|T\|^n =\omega(T)^n.
\]
\end{rem}

\begin{lem}\label{Commutant-Numerical}
Let $T$ be a shift on a Hilbert space $\clh$ and
let $R$ be a commutant of $T$. Then the numerical radius 
of $R$ is $\|R \|$.
\end{lem}

\begin{proof}
Let $T$ be a shift on a Hilbert space $\clh$. Then $T$
is unitarily equivalent to $M_z$ on $H^2_{\cle}(\D)$, where 
$\cle =Ker(T^*)$. That means there exists a unitary  $\Pi: \clh 
\rightarrow H^2_{\cle}(\D)$ such that
\[
\Pi T \Pi^* = M_z.
\]
Now if $R$ is any commutant of $T$, that is $RT=TR$, then
\[
(\Pi R \Pi^*)M_z = M_z(\Pi R \Pi^*).
\]
Thus $\Pi R \Pi^*$ is an analytic Toeplitz operator on $H^2_{\cle}(\D)$
and hence $\Pi R \Pi^* =M_{\Phi}$ for some $\Phi \in H^{\infty}_{\clb(\cle)}(\D)$.
Therefore
\[
\omega(R)= \omega(\Pi R \Pi^*) = \omega(M_{\Phi}) = \sup_{\lambda \in \D}\|\Phi^{*}(\lambda)\| 
=  \|M_{\Phi} \|= \|R \|.
\]
\end{proof}

The next result gives a partial answer of the question raised in the previous section.
\begin{thm}
Let $T_1, T_2 \in \clb(\clh)$ such that $T_1T_2=T_2T_1$ and $T_1T_2$ be a 
shift on $\clh$. Then $\omega(T_1T_2) \leq \omega(T_1)\omega(T_2)$.
\end{thm}

\begin{proof}
Let $T=T_1T_2$ be a shift on $\clh$ with $T_1T_2=T_2T_1$.
Since $T$ is a shift, Wold decomposition yields a unitary  $\Pi: \clh 
\rightarrow H^2_{\cle}(\D)$, where $\cle = Ker(T^{*})$ such that
\[
\Pi T \Pi^* = M_z.
\]
Now 
\[
TT_1 = (T_1T_2)T_1 =T_1(T_1T_2) = T_1T. 
\]
Similarly,
\[
TT_2 = T_2T. 
\]
Therefore from the above Lemma \ref{Commutant-Numerical}, we have
\[ 
\omega(T_i)=  \| T_i\| \quad (i=1, 2).
\]
Thus
\[
\omega(T)= \omega(T_1T_2) \leq \|T_1 \| \| T_2 \|= \omega(T_1)\omega(T_2).
\]
\end{proof}

Now we discuss some results on Berezin number of an operator on a reproducing 
kernel Hilbert space.

Let $H(K, \D)$ be a reproducing kernel Hilbert space with kernel function
$K$ defined as $K(\cdot, w)(z)= K(z, w)$ for all $z \in \D$. We say that
$H(K, \D)$ is an analytic Hilbert space if the multiplication operator 
$M_z$ by the coordinate function $z$ on $H(K, \D)$, defined by $M_zf = zf$ for all 
$f \in H(K, \D)$ is a contraction. Also
\[
M_z^*K(\cdot, w) =\bar{w}K(\cdot, w) \quad (w \in \D).
\]
Let $\hat{K}(\cdot, w)$ for $w \in \D$ be the normalized kernel function. 
Then
\begin{align*}
ber(M_z) & = \sup\{|\langle M_z \hat{K}(\cdot, w), \hat{K}(\cdot, w) \rangle|: w \in \D \}\\
& =\sup\{|\langle  \hat{K}(\cdot, w), M_z^*\hat{K}(\cdot, w) \rangle|: w \in \D \}\\
& =\sup\{|\langle  \hat{K}(\cdot, w), \bar{{w}}\hat{K}(\cdot, w) \rangle|: w \in \D \}\\
& =\sup\{|{{w}}| : w \in \D \}\\
& = 1.
\end{align*}
Again
\begin{align*}
1 \geq \omega(M_z) & =\sup\{|\langle M_z f, f \rangle|: \|f \|=1 \}\\
& \geq \sup\{|\langle M_z \hat{K}(\cdot, w), \hat{K}(\cdot, w) \rangle|: w \in \D \}\\
& = 1.
\end{align*}

To summarize the above, we have the following: 
\begin{thm}
Let $H(K, \D)$ be an analytic reproducing kernel Hilbert space over the
unit disc $\D$ with the kernel function $K$ and $M_z$ be the multiplication 
operator defined by the coordinate function $z$ on $H(K, \D)$. Then 
$ber(M_z)=1$ and $\omega(M_z)=1$. 
\end{thm}

Recall that $H(\Omega)$ be a rkHs on $\Omega$ and $\hat{k}_{\lambda}$ 
for $\lambda \in \Omega$ normalized kernel function of $H(\Omega)$. 
For $T \in \clb(H(\Omega))$, 
\[
\| T \|_{ber} = \sup_{\lambda \in \Omega} \| T \hat{k}_{\lambda} \|.  
\]
Here we have the following result:
\begin{thm}
Let $H(\Omega_i)$ be a reproducing kernel Hilbert space on $\Omega_i$ 
and $T_i \in \clb(H(\Omega_i))$ for $i=1, 2$.Then 
$\|T_1 \otimes T_2\|_{ber} =\|T_1 \|_{ber} \| T_2\|_{ber}$.
\end{thm}

\begin{proof}
Suppose $H(\Omega_i)$ is a reproducing kernel Hilbert space on $\Omega_i$ 
and $T_i \in \clb(H(\Omega_i))$ for $i=1, 2$.
Then $T_1 \otimes T_2 \in \clb(H(\Omega_1)\otimes H(\Omega_2))$.
Let $\hat{k}_{\lambda}$ for $\lambda \in \Omega_1$ and $\hat{k}_{\mu}$ for $\mu \in \Omega_2$
be the normalized kernel functions for $H(\Omega_1)$ and $H(\Omega_2)$, respectively.
Then $\hat{k}_{\lambda} \otimes \hat{k}_{\mu}$ for $\lambda \in \Omega_1$, $\mu \in \Omega_2$ 
is the normalized kernel function for $H(\Omega_1)\otimes H(\Omega_2)$. Therefore
\begin{align*}
\|T_1 \otimes T_2\|_{ber} & = 
\sup_{\lambda \in \Omega_1, \mu \in \Omega_2} \|(T_1 \otimes T_2)(\hat{k}_{\lambda} \otimes 
\hat{k}_{\mu}) \|\\
& = 
\sup_{\lambda \in \Omega_1, \mu \in \Omega_2} \|T_1 \hat{k}_{\lambda} \otimes T_2\hat{k}_{\mu} \|\\
& = 
(\sup_{\lambda \in \Omega_1} \|T_1 \hat{k}_{\lambda} \|) (\sup_{ \mu \in \Omega_2}\| T_2\hat{k}_{\mu} \|) \\
& = \|T_1 \|_{ber} \| T_2\|_{ber}. 
\end{align*}
\end{proof}

\begin{lem}
Let $H(\Omega)$ be any reproducing kernel Hilbert space on $\Omega$ and 
let $A, X \in \clb(\clh(\Omega))$ and $X$ be any positive operator.
Then $ ber({\sqrt{A^*XA}}) < \sqrt{ber(A^*XA})$.
\end{lem}

\begin{proof}
Suppose $A, X \in \clb(H(\Omega))$ and $X$ is positive operator. 
Then $A^{*}XA $ is a positive operator on $H(\Omega)$. 
Let $\hat{k}_{\lambda}$ for $\lambda \in \Omega$ be the normalized 
kernel function for $H(\Omega)$. Then
\begin{align*}
ber(\sqrt{A^*XA})) & 
= \sup_{\lambda \in \Omega} |\langle \sqrt{A^*XA} \hat{k}_{\lambda},  
\hat{k}_{\lambda} \rangle| \\
& < \sup_{\lambda \in \Omega} |\langle {A^*XA} \hat{k}_{\lambda},  
\hat{k}_{\lambda} \rangle|^{1 \over 2} \\
& =\sqrt{ber(A^*XA}).
\end{align*}
Here the second inequality comes from Lemma \ref{Macarthy-lemma}.
This proves the desired inequality.
\end{proof}

We will give some concrete operators for which Berezin number, numerical radius,
and Crawford number are calculated.

\begin{ex}
Let $\{e_1, e_2 \}$ be the standard orthonormal basis for $\mathbb{C}^2$.
Consider $\mathbb{C}^2$ as a rkHs on the set $\{1, 2\}$. 
Then $e_1, e_2$ (identified a complex pair as a function) are the kernel functions 
defined by 
\[  
e_i(j) = \begin{cases} 
      1 & ~\mbox{if}~i=j \\
      0 & ~\mbox{if}~i\neq j 
       \end{cases}
\] 
for $i, j \in \{1, 2\}$. Let 
$U:\mathbb{C}^2 \rightarrow \mathbb{C}^2$ be a linear map defined by
\begin{align*}
U(e_1) & = e_2 \\
U(e_2) & = e_1.
\end{align*}
Then clearly $U$ is a self-adjoint and unitary operator on
the rkHs $\mathbb{C}^2$. From the definition we have
\[
ber(U)= \sup\{|\langle Ue_1, e_1 \rangle|,  |\langle Ue_2, e_2 \rangle|\}= 0,
\]
\[
\|U \|_{ber} = \sup\{\| Ue_i \|: i =1, 2 \}= 1,
\]
and
\[
\omega(U) = \|U \|=1, ~~c(U)=0.
\]
\end{ex}

\begin{ex}\label{Example-shift}
Let $\{e_n : n\in \mathbb{Z}_{+} \}$ be a standard orthonormal basis 
for $\ell^2(\mathbb{Z}_{+})$ and $\bm{w}=\{w_n\}_{n \in \mathbb{Z}_{+}}$ 
be a bounded sequence of complex numbers. Consider $\ell^2(\mathbb{Z}_{+})$ 
as a rkHs on the set $\mathbb{Z}_{+}$. With the identification a sequence
as a function, the kernel function $e_n$ for $n\in \mathbb{Z}_{+}$ defined
as 
\[  
e_n(m) = \begin{cases} 
      1 & ~\mbox{if}~n=m \\
      0 & ~\mbox{if}~n \neq m 
       \end{cases}
\] 
for $m, n \in \mathbb{Z}_{+}$.
Let
$S_{\bm{w}}:\ell^2(\mathbb{Z}_{+}) \rightarrow \ell^2(\mathbb{Z}_{+})$ be 
a weighted right shift operator defined by
\begin{align*}
S_{\bm{w}}(x_0, x_1, x_2, \ldots ) & = (0, w_0x_0, w_1x_1, w_2x_2, \ldots ). 
\end{align*}
From the definition one can easily compute
\[
ber(S_{\bm{w}})= \sup\{|\langle S_{\bm{w}}e_i, e_i \rangle|: i \in \mathbb{Z}_{+}\}
= \sup\{|\langle w_ie_{i+1}, e_i \rangle|: i \in \mathbb{Z}_{+}\} = 0,
\]
\[
\|S_{\bm{w}} \|_{ber} = \sup\{ \| S_{\bm{w}}e_i \|: i \in \mathbb{Z}_{+} \}
= \sup_{n}|w_n|= \|S_{\bm{w}} \|.
\]
In particular, if we take the sequence $\bm{w}=\{w_n\}_{n \in \mathbb{Z}_{+}}$
as $w_n =1$ for all $n$. Then the operator $S_{\bm{w}}$ denoted as $S$, is called 
a unilateral shift on the rkHs $\ell^2(\mathbb{Z}_{+})$ over the set $\mathbb{Z}_{+} $. 
Then we have
\[
ber(S)= \sup\{|\langle Se_i, e_i \rangle|: i \in \mathbb{Z}_{+}\}
= \sup\{|\langle e_{i+1}, e_i \rangle|: i \in \mathbb{Z}_{+}\} = 0,
\]
\[
\|S \|_{ber} = \sup\{\| Se_i \|: i \in \mathbb{Z}_{+} \}= 1 = \|S\|
\]
and by Theorem \ref{Shift-Numerical} \& Remark \ref{Isometry-Numerical},
we have
\[
\omega(S) =1, ~~c(S)=0.
\]
\end{ex}

\begin{ex}
Consider $\ell^2(\mathbb{Z}_{+})$ as rkHs with kernel function
$e_n$ for $n\in \mathbb{Z}_{+}$ as before. Let $\{ \lambda_n \}$ be a bounded 
sequence of scalars such that $\lambda_n \rightarrow 0$ and
$D:\ell^2(\mathbb{Z}_{+}) \rightarrow \ell^2(\mathbb{Z}_{+})$ be 
a diagonal operator defined by $De_n  = \lambda_ne_n$ for 
$n\in \mathbb{Z}_{+}$. It is well known that $D$ is a compact 
normal operator. Now
\[
ber(D)= \sup\{|\langle De_i, e_i \rangle|: i \in \mathbb{Z}_{+}\}
= \sup\{|\langle \lambda_ie_{i}, e_i \rangle|: i \in \mathbb{Z}_{+}\}
= \sup_{i}{|\lambda_i|},
\]

\[
\|D \|_{ber} = \sup\{\| De_i \|: i \in \mathbb{Z}_{+} \}=  \sup_{i}{|\lambda_i|} = \|D \|.
\]
Also
\[
\omega(D) = \|D \|.
\]
and
\begin{align*}
0\leq c(D)& = \inf \{|\langle Dx, x \rangle|: x \in \ell^2(\mathbb{Z}_{+}), \|x \| =1 \} \\
    & \leq \inf\{|\langle De_i, e_i \rangle|: i \in \mathbb{Z}_{+}\}\\
    &= \inf\{|\langle \lambda_ie_{i}, e_i \rangle|: i \in \mathbb{Z}_{+}\}\\
    &= \inf_{i}{|\lambda_i|}= 0.
\end{align*}
Therefore $c(D)=0$.
\end{ex}

\begin{rem}
Let $T\in \clb(\clh)$ be a compact normal operator on an infinite dimensional
Hilbert space $\clh$. Then from the spectral theorem (cf.\cite{Conway-Book}), 
we have a countable non zero distinct eigenvalues $\{\lambda_n \}$ such that
$\lambda_n \rightarrow 0$ and a sequence of orthonormal eigenvectors $\{f_n \}$ such that
\[
Tx = \sum_{n=1}^{\infty}{\lambda_n}\langle x, f_n\rangle f_n \quad (x \in \clh).
\]
We have
\begin{align*}
 c(T)& = \inf \{|\langle Tx, x \rangle|: x \in \clh, \|x \| =1 \} \\
    & \leq \inf\{|\langle Tf_n, f_n \rangle|: n \in \mathbb{N}\}\\
    &= \inf\{|\langle \lambda_nf_n, f_n \rangle|: n \in \mathbb{N}\}\\
    &= \inf_{n}{|\lambda_n|}= 0.
\end{align*}
Therefore the Crawford number of a compact normal operator $T$ is zero.
\end{rem}

\begin{rem}
It is a noteworthy to mention that the unilateral shift $S$ on 
$\ell^2(\mathbb{Z}_{+})$ and the multiplication operator $M_z$ by the the
coordinate function $z$ on $H^2(\D)$ are unitarily equivalent. 
Therefore $\omega(S)=1=\omega(M_z)$. On the other hand let $\hat{k}(\cdot, w)$
for $w \in \D $ be the normalized Szeg\"{o} kernel function on $H^2(\D)$. 
Then
\begin{align*}
ber(M_z) & = \sup\{|\langle M_z \hat{k}(\cdot, w), \hat{k}(\cdot, w) \rangle|: w \in \D \}\\
& =\sup\{|{{w}}| : w \in \D \} = 1.
\end{align*}
But from the above Example \ref{Example-shift}, we have
\[
ber(S)=0.
\]
Therefore Berezin number does not satisfy the unitarily invariant property.
\end{rem}

\section{Hardy-Type Inequality for Berezin number}

In this section, we establish some Berezin number inequalities
of a certain class of operators with the help of classical Hardy 
inequality and functional calculus.

\begin{thm}{\label{Main-thm}}
Let $p>1$ be any real number and $N$ be any fixed natural number and
let $f$ be a positive continuous function defined on $\Delta \subset (0, \infty)$.
Let $A$ be a positive operator on a Hilbert space $\clh$ such that the
spectrum of $A$ lies in $\Delta$. Then for all $x\in \clh$
\[
\Big(1+ N^{p-1}\displaystyle\sum_{n= N+1}^{\infty}\frac{1}{n^p}\Big)
\langle f(A)x, x \rangle^p < \Big(\frac{p}{p-1}\Big)^{p} \langle f(A)^px, x \rangle.
\]
\end{thm}

\begin{proof}
Let $N$ be a fixed natural number and let $\{a_n\}$ be a sequence
of non-negative real numbers such that $a_n=0$ for $n> N$. Then the 
inequality (\ref{Hardy-Inequality}) yields
\begin{align}{\label{eqn-5}}
& a_1^p + (\frac{1}{2}\displaystyle\sum_{k=1}^{2}a_k )^p+\cdots+
(\frac{1}{N}\displaystyle\sum_{k=1}^{N}a_k )^p +(\displaystyle \sum_{k=1}^{N}a_k )^p
\{\displaystyle\sum_{n={N+1}}^{\infty}\frac{1}{n^p} \}
 < (\frac{p}{p-1} )^{p} (\sum_{k=1}^{N}a_{k}^p ).
\end{align}
Since $A \in \clb(\clh)$ is a positive operator and $f$ is a continuous 
function on $\Delta \subset (0, \infty)$ with the spectrum of $A$ lies in 
$\Delta$, we can use the continuous functional calculus. Therefore $f(A) 
\in \clb(\clh)$ makes sense and for any $x \in \clh$,  $\langle f(A)x, x 
\rangle \geq 0$. Now replacing $a_k$ by $\langle f(A)x, x \rangle$ for 
$k=1, 2, \ldots, N$ in the above inequality (\ref{eqn-5}), we obtain
\begin{align}{\label{eqn6}}
N \langle f(A)x, x \rangle^p + {N}^p\langle f(A)x, x \rangle^p
\Big\{\displaystyle\sum_{n= N+1}^{\infty}\frac{1}{n^p}\Big\} &
< N\Big(\frac{p}{p-1}\Big)^{p} \langle f(A)x, x \rangle^p.
\end{align}
Applying Lemma \ref{Macarthy-lemma} in the right hand side of 
the above inequality (\ref{eqn6}), we have
\begin{align*}
\langle f(A)x, x \rangle^p\Big(1+ N^{p-1}\displaystyle\sum_{n=N+1}^{\infty}
\frac{1}{n^p}\Big) & < \Big(\frac{p}{p-1}\Big)^{p} \langle f(A)^p x, x \rangle.
\end{align*}
This completes the proof.
\end{proof}

In particular, when $p=2$ and $N=2$, that is a sequence $\{a_n\}$
of non-negative real numbers such that $a_n=0$ for $n> 2$ then
from the above theorem we have the following result.

\begin{cor}
Let $f$ be a positive continuous function defined on $\Delta \subset (0, \infty)$
and $A$ be a positive operator on $\clh$ such that the spectrum of $A$ lies in
$\Delta$. Then for all $x\in \clh$
\[
\langle f(A)x, x \rangle^2 < \frac{24}{2\pi^2 -9}\langle f(A)^2x, x \rangle.
\]
\end{cor}

\begin{rem}
Now if we choose $N=3$, that is first three terms of the sequence $\{a_n\}$
are non-zero, and the rests are zero, then applying the similar techniques 
used in Theorem \ref{Main-thm}, the following power inequality is obtained:
\begin{center}
$\Big(2+3^p(\displaystyle\sum_{n=3}^{\infty}1/n^2)\Big)\langle f(A)x, x 
\rangle^p< 3(\frac{p}{p-1})^p ~\langle f(A)^px, x \rangle$.
\end{center}
In particular, for $p=2$, we deduce the following inequality
\begin{center}
$\langle f(A)x, x \rangle^2<\frac{48}{6\pi^2-37}~\langle f(A)^2x, x 
\rangle\approx 2.1604 ~\langle f(A)^2x, x \rangle$.
\end{center}
\end{rem}

\begin{rem}
For the sake of simplicity, we shall discuss operator inequalities 
for Berezin numbers by taking \textit{Hardy's inequality} for a 
scalar sequence $\{a_n \}$ with two non-zero positive numbers and 
rests are zero, i.e., $a_n =0$ for $n>2$.
\end{rem}

We now focus operators on reproducing kernel Hilbert spaces 
$H(\Omega)$ and we obtain the following inequalities for Berezin 
number with the help of Berezin transform.

\begin{thm}\label{Berezin-thm}
Let $p>1$ be any real number and let $f$ be a positive continuous 
function defined on $\Delta \subset (0, \infty)$. Let $A$ be a 
positive operator on a reproducing kernel Hilbert space $H(\Omega)$ 
such that the spectrum of $A$ lies in $\Delta$. Then the below inequality
\[
\Big(1+ 2^{p-1}\displaystyle
\sum_{n=3}^{\infty}\frac{1}{n^p}\Big)ber^p(f(A))< \Big(\frac{p}{p-1}\Big)
^{p}~ber(f(A)^p) ~holds~ true.
\]
In particular, we have the following inequality for $p=2$
\[
ber^2(f(A))< \frac{24}{2\pi^2 -9}~ber(f(A)^2).
\]
\end{thm}

\begin{proof}
Suppose $p>1$ is any real number and $\{a_n\}$ is a sequence of non-negative 
real numbers such that $a_n=0$ for $n> 2$. Then from the above Theorem 
\ref{Main-thm} for any $x \in H(\Omega)$, we have
\begin{align}{\label{eqn5}}
\Big(1+ 2^{p-1}\displaystyle\sum_{n= 3}^{\infty}\frac{1}{n^p}\Big)\langle 
f(A)x, x \rangle^p< \Big(\frac{p}{p-1}\Big)^{p} \langle f(A)^px, x \rangle.
\end{align}
Take $x=\hat{k}_\lambda$ for $\lambda \in \Omega$, in the above inequality, 
we get
\[
\Big(1+ 2^{p-1}\displaystyle\sum_{n= 3}^{\infty}\frac{1}{n^p}\Big)\langle 
f(A)\hat{k}_\lambda, \hat{k}_\lambda \rangle^p< \Big(\frac{p}{p-1}\Big)^{p}
\langle f(A)^p\hat{k}_\lambda, \hat{k}_\lambda \rangle.
\]
Hence, we have
\[
\Big(1+ 2^{p-1}\displaystyle\sum_{n= 3}^{\infty}\frac{1}{n^p}\Big)\langle f(A)
\hat{k}_\lambda, \hat{k}_\lambda \rangle^p< \Big(\frac{p}{p-1}\Big)^{p}
\sup_{\lambda \in \Omega } |\langle f(A)^p\hat{k}_\lambda, \hat{k}_\lambda \rangle|.
\]
Therefore by using the definition of Berezin number, we obtain the desired inequality.

Now if we put $p=2$, then the following inequality for Berezin number is obtained 
from the above
\begin{center}
$ber^2(f(A))< \frac{24}{2\pi^2-9} ~ber(f(A)^2)\approx 2.2349 ~ber(f(A)^2)$.
\end{center}
\end{proof}

\begin{rem}{\label{remyama}}
Using the Hilbert-Hardy inequality Yamanc{\i} and G\"{u}rdal \cite{YAMA} obtained
a strict inequality for Berezin number as follows:
\begin{center}
$ber^2(f(A))<\frac{36\pi^2}{53}~ber(f(A)^2)\approx 6.7039 ~ber(f(A)^2)$.
\end{center}
\end{rem}

\begin{rem}
Some reverse inequalities for Berezin numbers are obtained by 
Garayev et al.\cite{GARAGUR}. Indeed, 
they obtain the following inequality (see \cite{GARAGUR}, Proposition 1):
\begin{align*}
ber(f(A))^2 \leq \frac{3(8\pi -3)}{8} ber(f(A)^2).
\end{align*}  
Further in another result (in fact Proposition 2 of \cite{GARAGUR}), 
the authors presented the better approximate of the above result as follows:
\begin{align*}
ber(A)^2 \leq \frac{24\pi}{17} ber(A^2),
\end{align*}
where $A \in \clb(H(\Omega))$ is a self adjoint operator with spectrum 
contained in $ \Delta \subset(0, \infty)$.
\end{rem}

\begin{rem}
In another work of Garayev et al. \cite{GARASAL}, the following power 
inequality of Berezin number for a positive operator $A$ and a continuous 
function $f$ is obtained by using Hilbert-Hardy inequality which is as follows:
\begin{align*}
ber(f(A))^p \leq \alpha ber(f(A)^p),
\end{align*}where $\alpha\equiv \alpha(p, q)=2\Big[\big(\frac{3}{2}\big)^p+1\Big]
^{-1}(pq)^p>1$ and $\frac{1}{p}+\frac{1}{q}=1$ with $p, q>1$. Hence for 
$p=2$, we get
\begin{align*}
ber(f(A))^2 \leq \frac{128}{13}ber(f(A)^2).
\end{align*}
\end{rem}

\begin{rem}
Let $A\in \clb(\clh(\Omega))$ be a self-adjoint operator and the 
spectrum of $A$ contained in $\Delta\subset(0, \infty)$. Then for 
any positive continuous function $f$ defined on $\Delta$, Garayev 
et al. \cite{GARASAL} obtained an inequality given as follows:
\begin{align*}
ber(f(A))^2 \leq 4(pq-1)ber(f(A)^2),
\end{align*}
where $\frac{1}{p}+\frac{1}{q}=1$ with $p, q>1$.
\end{rem}

Before we proceed further we establish an inequality for Berezin number 
for operators by using the recently improved version of the \emph{discrete
Hardy's inequality} (\ref{Hardy-Inequality-Modified}). In fact the 
statement of our result reads as follows:

\begin{cor}{\label{prop1}}
Let $f$ be a positive continuous function defined on $\Delta \subset 
(0, \infty)$. Let $A$ be a positive operator on a reproducing kernel 
Hilbert space $H(\Omega)$ such that the spectrum of $A$ lies in $\Delta$.
Then the inequality
\[
ber^2(f(A)) < 1.593852~ber(f(A)^2)~ \mbox{holds}.
\]
\end{cor}

\begin{proof}
Using the improved version of the \emph{discrete Hardy inequality} 
(\ref{Hardy-Inequality-Modified}) in the proof of Theorem \ref{Berezin-thm}, 
we obtain the following inequality
\[
ber^2(f(A))< \displaystyle \frac{2}{w_1+ 4 \sum_{n=1}^{\infty} w_{n+1}} ber(f(A)^2),
\]
where for $n\in \mathbb{N}$, $w_n = 2-\sqrt{1+\frac{1}{n}}-\sqrt{1-\frac{1}{n}}$ 
and $\displaystyle \sum_{n=1}^{\infty}w_n=0.753045 ~(approx.)$. 
This proves the desired inequality.
\end{proof}


Let $H(\Omega)$ be a reproducing kernel Hilbert space and 
$A \in \clb(H(\Omega))$. Let $x \in H(\Omega)$ with $\|x \|=1$.
Using the Hardy's inequality (\ref{Hardy-Inequality}) and replacing
first two nonzero terms by $|\langle Ax, x \rangle|$ and rest of the
terms are zero, we obtain the following inequality:
\begin{align}{\label{eqn9}}
|\langle Ax, x \rangle|^p {(1+ 2^{p-1}\displaystyle\sum_{n=3}^{\infty}\frac{1}{n^p})}
<(\frac{p}{p-1} )^{p} |\langle Ax, x \rangle|^p.
\end{align}
Applying the mixed Cauchy-Schwartz inequality Lemma \ref{Convex-lemma} 
on the right hand side term, we get
\begin{align*}
|\langle Ax, x \rangle|^p
& \leq \Big( \langle |A|^{2\alpha}x, x \rangle^{\frac{1}{2}}
\langle |A^*|^{2(1-\alpha)}x, x \rangle^{\frac{1}{2}}\Big)^p\\ \nonumber
& \leq \Big(\frac{\langle |A|^{2\alpha} x, x \rangle +
\langle |A^*|^{2(1-\alpha)}x, x\rangle}{2}\Big)^p \\ \nonumber
& \leq \frac{1}{2}\Big(\langle |A|^{2\alpha}x, x \rangle^p +
\langle |A^*|^{2(1-\alpha)}x, x \rangle^p\Big) \\ \nonumber
& \leq \frac{1}{2}(\langle |A|^{2p\alpha}x, x \rangle +
\langle |A^*|^{2p(1-\alpha)}x, x \rangle ) \\ \nonumber
& = \frac{1}{2}\langle (|A|^{2p\alpha}+|A^*|^{2p(1-\alpha)}) x, x \rangle,
\end{align*}
where $\alpha \in [0, 1]$. Now take $x = \hat{k}_{\lambda}$, a normalized 
kernel function for $\lambda \in \Omega$, we have from the inequality 
(\ref{eqn9})
\[
|\langle A\hat{k}_{\lambda}, \hat{k}_{\lambda} \rangle|^p 
{(1+ 2^{p-1}\displaystyle\sum_{n=3}^{\infty}\frac{1}{n^p})} <
\frac{1}{2}(\frac{p}{p-1} )^{p} \langle (|A|^{2p\alpha}+|A^*|^{2p(1-\alpha)}) 
\hat{k}_{\lambda}, \hat{k}_{\lambda} \rangle
\]
From the definition of Berezin number inequality, we obtain
\[
ber^p(A) {(2+ 2^{p}\displaystyle\sum_{n=3}^{\infty}\frac{1}{n^p})}
< (\frac{p}{p-1} )^{p} \left(ber(|A|^{2p\alpha})+ber(|A^*|^{2p(1-\alpha)})\right).
\]

To summarize the above, we have the following Hardy-type Berezin number inequality:

\begin{thm}
Let $p>1$ be any real number and $A$ be any bounded operator
on a reproducing kernel Hilbert space $H(\Omega)$. Then the following
inequality
\[
ber^p(A) {(2+ 2^{p}\displaystyle\sum_{n=3}^{\infty}\frac{1}{n^p})}
<(\frac{p}{p-1} )^{p} \left(ber(|A|^{2p\alpha})+ber(|A^*|^{2p(1-\alpha)})\right)
~\mbox{holds}
\]
for $\alpha \in [0, 1]$.
\end{thm}

\begin{prop}
Let $p>1$ be any real number. Suppose that 
$f: \mathbb{R}\rightarrow (0, \infty)$ is a continuous
and convex function and $A$ is a self adjoint operator
on a reproducing kernel Hilbert space $H(\Omega)$. Then
\begin{align}{\label{eqn11}}
f^p(ber A)\Big(1+ 2^{p-1}\displaystyle\sum_{n=3}^{\infty}\frac{1}{n^p}\Big)
& < \Big(\frac{p}{p-1}\Big)^{p}ber(f^p(A)).
\end{align}In particular,
\begin{align}{\label{eqn12}}
f^2(ber A) < \frac{24}{2\pi^2-9}ber(f^2(A)).
\end{align}
\end{prop}

\begin{proof}
Suppose that $p>1$ is any real number and $\{a_n\}$ is a sequence
of non-negative real numbers such that $a_n=0$ for $n> 2$.
Then from the Hardy's inequality and replacing $a_1, a_2$ by
$f(t)$ for $t \in \mathbb{R}$, we get
\begin{align*}
f^p(t)\Big(1+ 2^{p-1}\displaystyle\sum_{n=3}^{\infty}\frac{1}{n^p}\Big)
 & < \Big(\frac{p}{p-1}\Big)^{p}f^p(t).
\end{align*}
Since $A$ is a self adjoint operator on $H(\Omega)$,
$\langle A\hat{k}_\lambda, \hat{k}_\lambda \rangle \in \mathbb{R}$, where
$\hat{k}_\lambda$ for $\lambda \in \Omega$ is the normalized kernel vector.
Now choose $t=\langle A\hat{k}_\lambda, \hat{k}_\lambda \rangle=\widetilde{A}(\lambda)$
in the above inequality and using the convexity of $f$, we have
\[
f^p(\widetilde{A}(\lambda))\leq \widetilde{f^p(A)}(\lambda)
\]
which yields
\begin{align*}
f^p(\widetilde{A}(\lambda))\Big(1+ 2^{p-1}\displaystyle\sum_{n=3}^{\infty}\frac{1}{n^p}\Big)
 & < \Big(\frac{p}{p-1}\Big)^{p}\widetilde{f^p(A)}(\lambda).
\end{align*}
This proves the desired inequality.

In particular for $p=2$, we obtain
\begin{align}
f^2(ber A) < \frac{24}{2\pi^2-9}ber(f^2(A)).\nonumber
\end{align}
\end{proof}
\begin{rem}
In \cite{YAMAGUR} (Corollary 2.2) and \cite{YAMAGARA} (Theorem 2.1 (ii)),
the authors obtained following inequalities for convex function:
\begin{align}
f^2(ber A) < \Big(\frac{12}{7}\pi-\frac{3}{14}\Big)ber(f^2(A)).\nonumber
\end{align} and
\begin{align}
f^2(ber A) < \frac{15}{4}ber(f^2(A)),\nonumber
\end{align}respectively.
\end{rem}

\begin{rem}
Using the improved discrete Hardy inequality (\ref{imph}) and
Proposition \ref{prop1}, we can obtain improved Hardy-type Berezin number
inequalities (\ref{eqn11}) and (\ref{eqn12}).
\end{rem}

\vspace{1cm}
\begin{thebibliography}{99}
\vspace{1cm}

\bibitem{Aronzajn-rkHs}
{\sc N. Aronzajn},
{\it Theory of reproducing kernels},
{Trans. Am. Math. Soc.} 68, 337-404 (1950)



\bibitem{BERE}
         {\sc F. A. Berezin},
         {\it Covariant and contravariant symbols for operators},
         {Math. USSR Izv.} {\bf6} (1972), 1117-1151.
         
         
\bibitem{BEREQUAN}
         {\sc F. A. Berezin},
         {\it Quantization},
         {Math. USSR Izv.} {\bf8} (1974), 1109-1163.
 
 
\bibitem{BS-Mapping}
{\sc C. A., Berger and  J. G., Stampfli},  
{\it  Mapping theorems for the numerical range},  
{Amer. J. Math}. 89, 1047-1055 (1967)

         
         
\bibitem{BD-Numerical-I}
{\sc F. F. Bonsall and J. Duncan}, 
{\it Numerical Ranges I}, 
{Cambridge University Press}, Cambridge, 1971.

         
         
\bibitem{COBURN}
         {\sc L. A. Coburn},
         {\it Berezin transform and Weyl-type unitary operators on the Bergman space},
         {Proc. Amer. Math. Soc.} {\bf140} (2012), 3445-3451.
         



\bibitem{Conway-Book}
{\sc J. B. Conway}, 
{\it A Course in Functional Analysis, 2nd edition}, 
{Graduate Texts in Mathematics}, vol. 96, Springer-Verlag, New York, 1990.



         
\bibitem{GARAGUR}
         {\sc M. T. Garayev, M. G\"{u}rdal and S. Saltan},
         {\it Hardy type inequality for reproducing kernel Hilbert space operators and related problems},
         {Positivity} {\bf21} (2017), 1615-1623.
         
         
\bibitem{GARASAL}
         {\sc M. T. Garayev, S. Saltan and D. Dogdu},
         {\it On the inverse power inequality for the Berezin number of operators},
         {J. Math. Ineq.} {\bf12(4)} (2018), 997-1003.

\bibitem{GR-Book}
{\sc K. E. Gustafson and D. K. M. Rao}, 
{\it Numerical Range}, 
{Springer}, New York, 1997        


\bibitem{Halmos}
{\sc P. R. Halmos}, 
{\it A Hilbert Space Problem Book}, 
{2nd ed., Springer}, New York, 1982.




\bibitem{HARDLITTPOLY}
         {\sc G. H. Hardy, J. E. Littlewood and G. Polya:}
         {\it Inequalities},
         {2nd Edition, Cambridge University Press}, 1967.
         
         
         
     
     
     
         
         

\bibitem{KATO-Notes}
{\sc  T. Kato}, 
{ \it Notes on some inequalities for linear operators}, 
{Math. Ann.} 125, 208--212 (1952)





 
 
\bibitem{KDPS-Partially}
{\sc Deepak K. D., D. Pradhan and J. sarkar},
{Partially isometric Toeplitz operators on the polydisc},
Bulletin of the London Mathematical Society (To appear).

      
         
\bibitem{KELLER}
         {\sc M. Keller, Y. Pinchover and F. Pogorzelski},
         {\it An improved discrete Hardy inequality},
         {Amer. Math. Monthly.} {\bf125(4)} (2018), 347-350.
         
         
         
         
\bibitem{KARA}
         {\sc M. T. Karaev},
         {\it Berezin symbol and invertibility of operators on the functional Hilbert space},
         {J. Funct. Anal.} {\bf238} (2006), 181-192.
         
         
\bibitem{KARA2}
         {\sc M. T. Karaev},
         {\it Reproducing Kernels and Berezin symbol techniques in various questions of operator theory},
         {Complex Anal. and Oper. Theory} {\bf7} (2013), 983-2018.
         


\bibitem{NF-Book}
{\sc B. Sz.-Nagy and C. Foias}, 
{\it Harmonic analysis of operators on Hilbert space},
{North-Holland, Amsterdam-London}, 1970.


         

\bibitem{MSS-PAIRS} 
{\sc A. Maji, J. Sarkar and T. R. Sankar}, 
{\it Pairs of commuting isometries, I,} 
{ Studia Math. { 248} (2019), no.~2, 171--189.}


\bibitem{Olofsson} 
{\sc A. Olofsson}, 
{\it A von Neumann Wold decomposition of two-isometries}, 
{Acta Sci. Math. (Szeged) 70 (2004) 715-726.}     
     
       
\bibitem{PR-rkHs-Book}
{\sc V. I. Paulsen and M. Raghupathi}, 
{ \it An Introduction to the Theory of Reproducing Kernel Hilbert Spaces}, 
{Cambridge Studies in Advanced Mathematics}, vol. 152, Cambridge University Press, Cambridge, 2016. 



\bibitem{PEARCY}
         {\sc C. Pearcy},
         {\it An elementary proof of the power inequality for the numerical radius},
         {Michigan Math. J.} {\bf13} (1966), 289-291.
         


\bibitem{SIMON-Trace}
{\sc B. Simon},
{ \it Trace Ideals and Their Applications}, 
{Camrbidge University Press}, Camrbidge (1979)



\bibitem{Toeplitz}
{\sc O. Toeplitz}, 
{\it Das algebraische analogon zu einem satze von Fejer}, 
{Math Zeit.}, 2 (1918), 187-197



\bibitem{Wintner}
{\sc Wintner}, 
{\it Spektraltheorie der unendlichen matrizen}, 
{Leipzig} (1930).



\bibitem{YAMA}
         {\sc U. Yamanc{\i} and M. G\"{u}rdal},
         {\it On numerical radius and Berezin number inequalities for reproducing kernel Hilbert space},
         {New York J. Math.} {\bf23} (2017), 1531-1537.
         
         
\bibitem{YAMAGUR}
         {\sc U. Yamanc{\i}, M. G\"{u}rdal and M. T. Garayev},
         {\it Berezin number inequalities for convex function in reproducing kernel Hilbert space},
         {Filomat} {\bf31(18)} (2017), 5711-5717.                

\bibitem{YAMAGARA}
         {\sc U. Yamanc{\i}, M. T. Garayev and \c{C}. Celik},
         {\it Hardy-Hilbert type inequality in reproducing kernel Hilbert space: 
         its application and related results},
         {Linear and Multilinear Algebra} {\bf67(4)} (2019), 830-842.


\end{thebibliography}
\end{document}